\documentclass[12pt,letterpaper]{article}
\usepackage[utf8]{inputenc}
\usepackage{amsmath,amsfonts,amsthm,amssymb,verbatim,etoolbox,color,mathtools}
\usepackage{hyperref}
\usepackage{authblk}
\usepackage{mathrsfs}
\usepackage{fullpage}
\usepackage{dsfont}

\usepackage{biblatex}
\bibliography{biblio}
\usepackage[title]{appendix}

\usepackage[most]{tcolorbox}

\newcommand{\lo}[1]{\mathcal{L}(#1)}

\def\E{\mathbb{E}}
\def\P{\mathbb{P}}
\newcommand{\bbf}{\mathbb{F}}
\newcommand{\bbn}{\mathbb{N}}
\newcommand{\bbz}{\mathbb{Z}}
\newcommand{\bbr}{\mathbb{R}}

\newcommand{\bbc}{\mathbb{C}}

\renewcommand{\hat}{\widehat}

\DeclareMathOperator{\supp}{supp}

\renewcommand{\subset}{\subseteq}
\renewcommand{\supset}{\supseteq}

\newtheorem{lemma}{Lemma}[section]
\newtheorem{theorem}[lemma]{Theorem}

\newtheorem{proposition}[lemma]{Proposition}

\theoremstyle{definition}

\newtheorem{definition}[lemma]{Definition}
\newtheorem{Question}[lemma]{Question}

\date{}
\title{Improved Bounds for 3-Progressions}
\author{Rushil Raghavan}
\date{\small{Department of Mathematics, University of California, Los Angeles (UCLA) 
\\ 
e-mail: \href{mailto:rushil@ucla.edu}{rushil@ucla.edu}}}
\begin{document}
	\maketitle
\begin{abstract}
    We prove that if $A\subset \{1,\dots,N\}$ has no nontrivial three-term arithmetic progressions, then $|A|\leq \exp(-c\log(N)^{1/6}\log\log(N)^{-1})N$ for some absolute constant $c>0$. To obtain this bound, we use an iterated variant of the sifting argument of Kelley and Meka, as well as an improved bootstrapping argument for Croot-Sisask almost-periodicity due to Bloom and Sisask.
\end{abstract}

\section{Introduction}
A three-term arithmetic progression is a set of the form $\{a,a+d,a+2d\}.$ We say a progression is nontrivial if $d\neq 0$. We study the following question, asked by Erd\H{o}s and Tur\'an \cite{ET} in 1936: 

\begin{Question}\label{mainquestion} What is the largest size of a set $A\subset\{1,\dots,N\}$ containing no nontrivial three-term arithmetic progressions?
\end{Question}

In 2023, Kelley and Meka \cite{KM} made a significant breakthrough, giving a quasipolynomial bound for this problem:
\begin{theorem}\label{KM-3APbound} If $A\subset\{1,\dots,N\}$ contains no nontrivial three-term arithmetic progressions, then $|A|\leq \exp(-\Omega((\log N)^{1/12}))N$.
\end{theorem}

Prior to this, the best known bound was due to Bloom and Sisask \cite{BS3}, and was of the form $|A|\leq N/\log(N)^{1+\Omega(1)}$. 

One reason why Kelley and Meka's result is so significant is that the quasipolynomial shape of their bound matches the best known lower bound: the largest known subset $A\subset \{1,\dots,N\} $ containing no nontrivial three-term arithmetic progressions has size $\exp(-c\log(N)^{1/2})N$. A bound of this form was first proved by Behrend \cite{doi:10.1073/pnas.32.12.331}, and recently Elsholtz, Hunter, Proske, and Sauermann \cite{EHPS} have improved the implicit constant $c$. 

In view of these two estimates, the study of upper bounds in Question \ref{mainquestion} must now focus on increasing the exponent $1/12$. Later in 2023, Bloom and Sisask optimized the Kelley-Meka argument to produce the best known bound:
\begin{theorem}\label{BS-3APbound}  If $A\subset\{1,\dots,N\}$ contains no nontrivial three-term arithmetic progressions, then $|A|\leq \exp(-\Omega((\log N)^{1/9}))N$.
\end{theorem}

The main result of this paper is a further improved bound for Question \ref{mainquestion}.

\begin{theorem}\label{3APbound} Let $A\subset \{1,\dots,N\}$ have no nontrivial three-term arithmetic progressions. Then $|A|\leq \exp(-\log(N)^{1/6-o(1)})N$. More specifically, $|A|\leq \exp(-\Omega(\log(N)^{1/6}\log\log(N)^{-1/6}))N$.
\end{theorem}

As is typical in additive combinatorics, there are technical aspects of the proof of Theorem \ref{3APbound} that are made much simpler if one considers the group $\bbf_q^n$ instead of $\{1,\dots,N\}$. We thus also prove the following.

\begin{theorem}\label{ff-3APbound} Let $q\geq 3$ be prime. Let $A\subset \bbf_q^n$ have no nontrivial three-term arithmetic progressions. Let $N=q^n$. Then $|A|\leq \exp(-\Omega(\log(N)^{1/5}))N$.
\end{theorem}

This, by contrast, is not a new result, since better estimates can be obtained by the polynomial method \cite{EG}, which has no known analogue for the problem over $\{1,\dots,N\}$. 

\subsection*{Acknowledgements} We would like to thank Terence Tao for many helpful conversations and guidance. We would also like to thank Thomas Bloom and Olof Sisask for helpful conversations and for introducing the author to the weighted sifting argument in Lemma \ref{weightedsift}. We would also like to thank Fred Tyrrell and Mark Lewko for helpful comments on previous versions of this manuscript. The author is partially supported by NSF Grant DMS-2347850.

\section{Notation and Strategy of Proof}
\subsection{Definitions and Notation}
The conventions used in this paper are essentially the same as those in \cite{BS1}. We record the required definitions here. Throughout, $G$ will be a finite abelian group. To begin, for a quantity $\alpha\in(0,1]$, we let $\lo{\alpha} = \log(2/\alpha)$. Next, we record the various analytic quantities required for our argument.

\begin{definition}[Expectations, Inner Products, Measures, $L^p$ norms]
\leavevmode
\begin{itemize}  
    \item Given a function $f:G\to\bbr$, we define $\E_{x}f(x) = \frac{1}{|G|}\sum_{x\in G}f(x)$.
    \item Given a finite group $G$ and a nonnegative function $\mu:G\to \bbr$ with $\E_x\mu(x)=1$, we define 
\[\langle f,g\rangle_\mu = \E_xf(x)g(x)\mu(x)\quad\text{and}\quad \|f\|_{p(\mu)}= \left(\E_{x\in G}|f(x)|^p\mu(x)\right)^{1/p}\text{ for }p\in[1,\infty).\]
We call such functions $\mu$ probability measures on $G$. When no subscript $\mu$ is written, it is assumed that $\mu$ is identically $1$, i.e., the quantities above are taken with respect to the uniform probability measure on $G$.
\item Given functions $f,g:G\to\bbr$, we define the convolution and difference convolution
\[f\ast g(x) = \E_yf(y)g(x-y)\quad\text{and}\quad f\circ g(x)=\E_yf(y)g(x+y).\]
For an integer $k\geq 1$, we will let $f^{(k)}$ denote the $k$-fold convolution of $f$, i.e., $f\underbrace{\ast\dots\ast}_{k\text{ times}}f$.
\item Given a subset $A\subset G$, $\mu_A = \frac{|G|}{|A|}\mathds{1}_A$ is the indicator function of $A$, normalized to have $\|\mu_A\|_1=1$. 
\end{itemize}   
\end{definition}

\begin{definition}[Density, Relative Density] Given $A\subset G$, the quantity $\frac{|A|}{|G|}$ is the density of $A$. In a slight abuse of notation, we will also write this quantity as $\mu(A)$. Given sets $A\subset B\subset G$, the relative density of $A$ in $B$ is $\frac{|A|}{|B|}$.  
\end{definition}

\begin{definition}[Fourier Transform]  Given $f:G\to\bbr$, $\hat{G}$ is the group of homomorphisms from $G$ to $\bbc^\times$. We define $\hat{f}:\hat{G}\to\bbc$ by 
\[\hat{f}(\gamma)=\E_x f(x)\overline{\gamma(x)}.\]
\end{definition}

We record a number of standard facts about these quantities for use later on.
\begin{proposition}
Let $f,g,h:G\to\bbr$, and let $\mu$ be a probability measure on $G$. Then all the following are true:
\begin{itemize}
    \item $\langle f\ast g,h\rangle  = \langle f,h\circ g\rangle$.
    \item If $p\leq p'$, then $\|f\|_{p(\mu)}\leq \|f\|_{p'(\mu)}$.
    \item If $f\geq 0$, then $\|f\|_{1(\mu)} = \langle f,\mu\rangle$.
    \item $\hat{f\ast g}=\hat{f}\cdot\hat{g}$ and $\hat{f}\circ\hat{g}=\hat{f}\cdot\overline{\hat{g}}$.
    \item For $x\in G$, $f(x)=\sum_{\gamma\in\hat{G}}\hat{f}(\gamma)\gamma(x)$.
    \item $\langle f,g\rangle = \sum_{\gamma\in\hat{G}}\hat{f}(\gamma)\overline{\hat{g}(\gamma)}.$
\end{itemize}

\end{proposition}

Finally, we will use standard asymptotic notation.
\begin{definition} Given functions $f,g:[1,\infty)\to [0,\infty)$, we say $f = O(g)$ or $f\lesssim g$ if there is an absolute constant $C$ with $f\leq Cg$. We say $f=\Omega(g)$ or $f\gtrsim g$ if $g = O(f)$. We say $f = o(g)$ if for all $c>0$, there is an $N\geq 1$ such that for all $x\geq N$, $f(x)\leq cg(x)$. 
Finally, given functions $f,g,h:[1,\infty)\to\bbr$, we say $f = g+O(h)$ if $|f-g|=O(h)$.
\end{definition}

\subsection{Strategy of Proof}
Our proof follows the strategy of Kelley and Meka \cite{KM}. In order to discuss our modifications of their argument, we first present a variation of the summary provided by Bloom and Sisask in \cite{BS1}. Here, suppose $A\subset G$ has density $\alpha$ and no nontrivial three-term progressions. We can establish that $A$ has a density increment on a translate of a subspace (or Bohr set) as follows. In order to focus on the main ideas, we are intentionally vague about some of the parameters. 

\begin{enumerate}
    \item H\"older lifting and unbalancing: $\|\mu_A\circ\mu_A\|_p$ is large for some $p = O(\lo{\alpha})$.
    \item Sifting: Let $t$ be a threshold, and let $S = \{x:\mu_A\circ\mu_A(x)\geq t\}$. There are sets $A_1,A_2$, each with density $\alpha^{O(p)}$, such that $\langle \mu_{A_1}\circ\mu_{A_2},\mathds{1}_S\rangle$ is large. 
    \item Croot-Sisask almost-periodicity: There is a (somewhat large) set $X$ such that \\
    $\langle \mu_X^{(k)}\ast \mu_{A_1}\circ\mu_{A_2},\mathds{1}_S\rangle$ is large. In particular, $\langle \mu_X^{(k)}\ast \mu_{A_1}\circ\mu_{A_2},\mu_A\circ\mu_A\rangle$ is large.
    \item Bootstrapping: there is a structured set (a subspace or a Bohr set) $V$ such that \\ 
    $\langle \mu_V\ast \mu_X^{(k)}\ast \mu_{A_1}\circ\mu_{A_2},\mu_A\circ\mu_A\rangle$ is large.
    \item Density increment: $\|\mu_V\ast\mu_A\|_\infty$ is large, so $A$ has a density increment on a translate of $V$.
\end{enumerate}

We recall in more detail the sifting step, as given in \cite[Lemma 8]{BS1}.
\begin{lemma}[Sifting]\label{sifting} Let $p\geq 1$ be an integer and $\epsilon,\delta>0$. Let $C_1,C_2\subset G$ and let $\mu=\mu_{C_1}\circ\mu_{C_2}$. For any $A\subset G$ with density $\alpha$, if 
\[S = \{x\in G:\mu_A\circ\mu_A(x)>(1-\epsilon)\|\mu_A\circ\mu_A\|_{p(\mu)}\},\]
then there are $A_1\subset C_1$and $A_2\subset C_2$ such that 
\[\min\left(\frac{|A_1|}{|C_1|},\frac{|A_2|}{|C_2|}\right) \gtrsim \left(\alpha \|\mu_A\circ\mu_A\|_{p(\mu)}\right)^{2p+O_{\epsilon,\delta}(1)}\]
and 
\[\langle \mathds{1}_S,\mu_{A_1}\circ\mu_{A_2}\rangle \geq 1-\delta.\]
\end{lemma}

In \cite{BS2}, Bloom and Sisask give a quantitative improvement to the Kelley-Meka bound by improving step (4). There, they remark that one way to improve the bootstrapping procedure even further is to obtain upper bounds on the quantity 
\[\langle \mu_A\circ\mu_A,\mu_{A_1}\circ\mu_{A_1}\rangle^{1/2}\langle \mu_A\circ\mu_A,\mu_{A_2}\circ\mu_{A_2}\rangle^{1/2}.\]
Taking the Fourier transform and using the Cauchy-Schwarz inequality, we observe that
\[\langle \mu_A\circ\mu_A,\mu_{A_1}\circ\mu_{A_2}\rangle\leq \langle \mu_A\circ\mu_A,\mu_{A_1}\circ\mu_{A_1}\rangle^{1/2}\langle \mu_A\circ\mu_A,\mu_{A_2}\circ\mu_{A_2}\rangle^{1/2}.\]
In view of this, a nearly optimal input to the bootstrapping procedure would be an inequality of the form 
\begin{equation}\label{goodbootstrapineq} 
\langle \mu_A\circ\mu_A,\mu_{A_1}\circ\mu_{A_1}\rangle^{1/2}\langle \mu_A\circ\mu_A,\mu_{A_2}\circ\mu_{A_2}\rangle^{1/2} \leq C\langle \mu_A\circ\mu_A,\mu_{A_1}\circ\mu_{A_2}\rangle,\end{equation}

where $C$ is some absolute constant. 
If this inequality fails, say \[\langle \mu_A\circ\mu_A,\mu_{A_1}\circ\mu_{A_1}\rangle\geq 8\langle \mu_A\circ\mu_A,\mu_{A_1}\circ\mu_{A_2}\rangle,\]
then we can use the identity $\langle \mu_A\circ\mu_A,\mu_{A_1}\circ\mu_{A_1}\rangle = \|\mu_A\circ\mu_A\|_{1(\mu_{A_1}\circ\mu_{A_1})}$ and feed this lower bound into Lemma \ref{sifting}. We would obtain an even sparser pair of sets $A_3, A_4$, and an even higher level set, say $S'=\{x:\mu_A\circ\mu_A(x)\geq 4t\}$, such that $\langle \mu_{A_3}\circ\mu_{A_4},\mathds{1}_{S'}\rangle$ is large. In turn, we would have 
\[\langle\mu_A\circ\mu_A,\mu_{A_3}\circ\mu_{A_4}\rangle \geq 2 \langle\mu_A\circ\mu_A,\mu_{A_1}\circ\mu_{A_2}\rangle.\]
If the inequality (\ref{goodbootstrapineq}) still fails, we can iterate this. Since $\|\mu_A\circ\mu_A\|_\infty\leq \alpha^{-1}$, this can be iterated at most $O(\lo{\alpha})$ many times until the inequality (\ref{goodbootstrapineq}) is true. At each stage, the densities of the sets $A_{2j-1}$, $A_{2j}$ decrease by a factor of $\alpha^{O(1)}$, so even though they get sparser during this iteration, their densities will still be on the order of $\alpha^{O(\lo{\alpha})}$. We thus obtain quantitative savings in the bootstrapping procedure, without any significant losses in the other parameters.

The above summary is essentially all that is needed in the case $G=\bbf_q^n$, but additional complications arise when working relative to Bohr sets in $G=\bbz/N\bbz$. The most significant such obstacle is that in order to perform the almost-periodicity step (3), one needs $A_{2j-1}$ and $A_{2j}$ to be localized to different Bohr sets, one narrower than the other. Thus, it is unacceptable to sift relative to a measure like $\mu_{A_1}\circ\mu_{A_1}$. This can be circumvented using a probabilistic argument; the details are given in Lemma \ref{iteratedsifting}. Along the way, one needs slightly finer control over the densities produced by Lemma \ref{sifting}. To obtain this, we introduce a weighted variant of sifting in Lemma \ref{weightedsift}, which may be of independent interest.

Finally, we remark that the reciprocals of the exponents in Theorems \ref{3APbound} and \ref{ff-3APbound} differ by only $1+o(1)$ instead of $2$. A common feature of density increment arguments using Bohr sets such as those in \cite{BS2} and \cite{SS} is that at every step of the iteration, the radius of the Bohr set is dilated by a factor of $\alpha^{O(1)}$. Given that we have about $O(\lo{\alpha})$ iterations, this should result in a ``loss of two logarithms" between the bounds over $\bbf_q^n$ and $\bbz/N\bbz$ (cf. Lemma \ref{bohrsizebound}). To obtain Theorem \ref{3APbound}, we need to  obtain careful control of the radii of the Bohr sets involved in the density increment argument, and some reworking of the remainder of the argument is required to facilitate this. It is not clear whether a bound for Question \ref{mainquestion} of size $\exp(-\log(N)^{1/2-o(1)})$, nearly matching Behrend's example, could possibly be proved using a density increment involving Bohr sets. If it is possible, there would also need to be a loss of one logarithm between the $\bbf_q^n$ and $\bbz/N\bbz$ cases. See also \cite{S} for another density increment argument with comparably efficient control of the radii of Bohr sets.

\section{The Finite Field Case}\label{ff-mainsection}
In this section, we will prove Theorem \ref{ff-3APbound}. Throughout this section, we will let $G=\bbf_q^n$, where $q\geq 3$ is a prime and $N=q^n$. Theorem \ref{ff-3APbound} will follow from iterating the following proposition.

\begin{proposition}\label{prop-ff-it} if $A\subset G$ has density $\alpha$, then at least one of the following is true.
\begin{enumerate}
    \item $|\langle \mu_A\ast\mu_A,\mu_{-2\cdot A}\rangle - 1|\leq 2^{-1}$, or 
    \item There is a subspace $V$ with codimension $O(\lo{\alpha}^4)$ such that $\|\mu_A\ast \mu_V\|_\infty \geq 1+2^{-5}$. 
\end{enumerate}
\end{proposition}

Before proceeding with the proof of this proposition, we will record the proof of Theorem \ref{ff-3APbound}.
\begin{proof}[Proof of Theorem \ref{ff-3APbound}, assuming Proposition \ref{prop-ff-it}] Let $C>0$ be an absolute constant to be chosen later. Let $\alpha$ be the density of $A$ in $G$. Let $J$ be a maximal positive integer such that there exist subspaces $G=V_0\supset V_1\supset\dots\supset V_J$ and translates $0=x_0,x_1,\dots,x_J$ such that:

\begin{itemize}
    \item For all $j\geq 1$, the relative density of $(A+x_j)\cap V_j$ in $V_j$ is at least $(1+2^{-5})$ times the relative density of $(A+x_{j-1})\cap V_{j-1}$ in $V_{j-1}$, and 
    \item For each $j\geq 1$, the codimension of $V_{j}$ is at most $C\lo{\alpha}^4$ times the codimension of $V_{j-1}$.
\end{itemize}
Observe that by the first bullet point, we must have $J=O(\lo{\alpha})$, since the relative density of a set in another set can never exceed $1$.

We now apply Proposition \ref{prop-ff-it} with $V_J$ in place of $G$ and $(A+x_J)\cap V_J$ in place of $A$. Observe that the relative density of $(A+x_J)\cap V_J$ in $V_J$ is at least $\alpha$. If the constant $C$ is chosen large enough, then we cannot be in case (2) of Proposition \ref{prop-ff-it}, so we must be in case (1). This implies that the number of three-term progressions in $A+x_J$ is at least 
\[\frac{1}{2}\alpha^3|V_J|^2\geq \frac{1}{2}\alpha^3\exp(-O(\lo{\alpha}^5))N^2.\]
The number of trivial three-term progressions in $A$, on the other hand, is $\alpha N$. We thus must have 
\[N\leq 2\alpha^{-2}\exp(O(\lo{\alpha}^5)) = \exp(O(\lo{\alpha}^5)),\]
which rearranges to the desired inequality.
\end{proof}

We will follow the five-step procedure outlined in Section 2. We begin by recording the conclusion of the H\"older lifting and unbalancing steps from \cite{BS1}. 
\begin{lemma}\label{ff-holderunbalance} Let $A\subset G$ have density $\alpha$. Then if $|\langle \mu_A\ast\mu_A,\mu_{-2\cdot A}\rangle - 1|\geq 2^{-1}$, there is an exponent $p = O(\lo{\alpha})$ such that $\|\mu_A\circ\mu_A\|_{p}\geq 1+2^{-3}$. 
\end{lemma}

We now record the main application of sifting that we will use in our argument.

\begin{lemma}\label{ff-iteratedsifting} Let $p\geq 1$ be an integer. Let $A\subset G$ have density $\alpha$ and $\|\mu_A\circ\mu_A\|_{p}\geq 1+2^{-3}$. Then there is a parameter $\sigma\in [1+2^{-3},\alpha^{-1}]$ and there are sets $A_1,A_2\subset G$, each with density $\alpha^{O(p+\lo{\alpha})}$ in $G$, such that if \[S = \{x:\mu_A\circ\mu_A(x)\geq (1-2^{-7})\sigma\},\] 
then
\[\langle \mathds{1}_S,\mu_{A_1}\circ\mu_{A_2}\rangle \geq 1-2^{-7},\]
and for each $i\in\{1,2\}$,
\[\langle \mu_A\circ\mu_A,\mu_{A_i}\circ\mu_{A_i}\rangle \leq 2\sigma.\]
\end{lemma}
\begin{proof} The proof will proceed by iterating Lemma \ref{sifting}. We apply Lemma \ref{sifting} with $C_1=C_2=G$ and $\epsilon=\delta=2^{-7}$ to obtain sets $A_1^1,A_2^1\subset G$, each with density $\alpha^{O(p)}$, such that, if 
\[S^1 = \{x:\mu_A\circ\mu_A(x)\geq (1-2^{-7})\|\mu_A\circ\mu_A\|_{p}\},\]
then $\langle \mathds{1}_{S^1},\mu_{A_1^1}\circ\mu_{A_2^1}\rangle \geq 1-2^{-7}.$
If, for each $i\in\{1,2\}$, $\langle \mu_A\circ\mu_A,\mu_{A_i^1}\circ\mu_{A_i^1}\rangle \leq 2\|\mu_A\circ\mu_A\|_{p}$, then we are done with $A_1=A_1^1$, $A_2=A_2^1$, and $\sigma=\|\mu_A\circ\mu_A\|_p$. Otherwise, if this inequality fails for some $i\in\{1,2\}$, we have \[\|\mu_A\circ\mu_A\|_{1(\mu_{A_i^1}\circ\mu_{A_i^1})} = \langle \mu_A\circ\mu_A,\mu_{A_i^1}\circ\mu_{A_i^1}\rangle\geq 2\|\mu_A\circ\mu_A\|_p.\] 
Let $c_1,c_2>0$ be constants to be chosen later. Let $A^1 = A_i^1$, and let $A^1,A^2,A^3,\dots,A^J$ be a maximal sequence of subsets of $G$ such that, for all $j\in\{1,\dots,J-1\}$,
\begin{itemize}
    \item $A^{j+1}\subset A^j$,
    \item $\frac{|A^{j+1}|}{|A^j|}\geq c_1\alpha^{c_2}$, and 
    \item $\|\mu_A\circ\mu_A\|_{1(\mu_{A^{j+1}}\circ\mu_{A^{j+1}})}\geq 2\|\mu_A\circ\mu_A\|_{1(\mu_{A^{j}}\circ\mu_{A^{j}})}$.
\end{itemize}
Since $\mu_A\circ\mu_A(x)\leq \alpha^{-1}$ for all $x\in G$, $\|\mu_A\circ\mu_A\|_{1(\mu)}\leq \alpha^{-1}$ for any probability measure $\mu$. Since $\|\mu_A\circ\mu_A\|_{1(\mu_{A^1}\circ\mu_{A^1})}\geq 1$, we must have $J = O(\lo{\alpha})$. 
We now apply Lemma \ref{sifting} with $C_1=C_2=A^J$, $p=1$, and $\epsilon=\delta=2^{-7}$. Provided $c_1$ was chosen small enough and $c_2$ large enough, we can find sets $A_1$ and $A_2$, each with relative density at most $c_1\alpha^{c_2}$ in $A^J$, such that if $S = \{x:\mu_A\circ\mu_A(x)\geq (1-2^{-7})\|\mu_A\circ\mu_A\|_{1(\mu_{A^J}\circ\mu_{A^J})}\}$, then 
\[\langle \mathds{1}_S,\mu_{A_1}\circ\mu_{A_2}\rangle \geq 1-2^{-7}.\]
Moreover, by the maximality of $(A^1,\dots,A^J)$, for each $i\in\{1,2\},$ we must have 
\[2\|\mu_A\circ\mu_A\|_{1(\mu_{A^J}\circ\mu_{A^J})} \geq \|\mu_A\circ\mu_A\|_{1(\mu_{A_i}\circ\mu_{A_i})} = \langle \mu_A\circ\mu_A,\mu_{A_i}\circ\mu_{A_i}\rangle. \]
Finally, for each $i\in\{1,2\}$, we have 
\[\frac{|A_i|}{|G|} = \frac{|A_i|}{|A^J|}\cdot \left(\prod_{j=1}^{J-1}\frac{|A^{j+1}|}{|A^{j}|}\right)\cdot \frac{|A^1|}{|G|} \geq \alpha^{O(\lo{\alpha}+p)}.\]
We have thus proved the proposition with $\sigma = \|\mu_A\circ\mu_A\|_{1(\mu_{A^J}\circ\mu_{A^J})}$. 
\end{proof}

 We record the almost-periodicity result we will need to complete the argument. This version is a special case of \cite[Theorem 3.2]{SS}.
\begin{theorem}[$L^\infty$ almost-periodicity]\label{ff-AP} Let $\epsilon>0$ and $k\geq 1$. Let $S\subset G$, $A_1\subset G$, $A_2\subset G$, and suppose $A_1$ and $A_2$ have densities $\alpha_1$ and $\alpha_2$,respectively. There is a set $X\subset G$ of size
\[|X|\gtrsim \exp(-O(\epsilon^{-2}k^2\lo{\alpha_1}\lo{\alpha_2}))|G|\]
such that 
\[\|\mu_X^{(k)}\ast \mu_{A_1}\circ\mu_{A_2}\ast \mathds{1}_S-\mu_{A_1}\circ\mu_{A_2}\ast\mathds{1}_S\|_\infty\leq \epsilon.\]
\end{theorem}

We will also need Chang's lemma. This version is a special case of \cite[Lemma 4.36]{TV}.
\begin{lemma}\label{ff-Chang} Let $X\subset G$ have density $\xi$. Let $\Delta_{1/2}(X) = \{\gamma\in G:|\hat{\mu_X}(\gamma)|\geq 1/2\}$. Then the subspace of $V$ orthogonal to all characters in $\Delta_{1/2}(X)$ has codimension $O(\lo{\xi})$.
\end{lemma}

With these prelimiary results recorded, we are now ready to prove Proposition \ref{prop-ff-it}.
\begin{proof}[Proof of Proposition \ref{prop-ff-it}] Let $A\subset G$ have density $\alpha$ and suppose that we are not in case (1). By Lemma \ref{ff-holderunbalance}, there is an exponent $p = O(\lo{\alpha})$ such that $\|\mu_A\circ\mu_A\|_p\geq 1+2^{-3}$. By Lemma \ref{ff-iteratedsifting}, we may find sets $A_1,A_2\subset G$, each having density $\alpha^{O(p)}$, and a parameter $\sigma$ with $1+2^{-3}\leq \sigma\leq \alpha^{-1}$ such that if 
\[S = \{x:\mu_A\circ\mu_A(x)\geq (1-2^{-7})\sigma\},\] 
then
\[\langle \mathds{1}_S,\mu_{A_1}\circ\mu_{A_2}\rangle \geq 1-2^{-7},\]
and for each $i\in\{1,2\}$,
\[\langle \mu_A\circ\mu_A,\mu_{A_i}\circ\mu_{A_i}\rangle \leq 2\sigma.\]
Let $k\geq 1$ be an integer to be chosen later. By Theorem \ref{ff-AP}, there is a set $X$ with density at least $\exp(-O(k^2\lo{\alpha_1}\lo{\alpha_2})) = \exp(-O(k^2\lo{\alpha}^4))$
such that 
\[\langle \mathds{1}_S,\mu_{A_1}\circ\mu_{A_2}\ast \mu_X^{(k)}\rangle \geq 1-2^{-6}.\]
We then have, by the definition of $S$,
\[\langle \mu_A\circ\mu_A,\mu_{A_1}\circ\mu_{A_2}\ast \mu_X^{(k)}\rangle \geq (1-2^{-5})\sigma.\]
Let $F = \mu_A\circ\mu_A\circ(\mu_{A_1}\circ\mu_{A_2})$. Let $V$ be the subspace of $G$ orthogonal to the characters in $\Delta_{1/2}(X)$, so that by Lemma \ref{ff-Chang}, the codimension of $V$ is $O(k^2\lo{\alpha}^4)$. We have, for any $t\in V$, 
\begin{align*}\|\mu_X^{(k)}\ast F(\cdot + t) - \mu_X^{(k)}\ast F\|_\infty &\leq \sum_{\gamma\in \hat{G}}|\hat{\mu_X}(\gamma)|^k|\hat{F}(\gamma)||\gamma(t)-1| \\
&\leq 2\sum_{\gamma\notin \Delta_{1/2}(X)}|\hat{\mu_X}(\gamma)|^k|\hat{F}(\gamma)| \\
&\leq 2^{1-k} \sum_{\gamma\in\hat{G}}|\hat{F}(\gamma)|\\
&\leq 2^{1-k}\left(\sum_{\gamma\in\hat{G}}|\hat{\mu_A}(\gamma)|^2|\hat{\mu_{A_1}}(\gamma)|^2\right)^{1/2}\left(\sum_{\gamma\in\hat{G}}|\hat{\mu_A}(\gamma)|^2|\hat{\mu_{A_2}}(\gamma)|^2\right)^{1/2} \\
&=2^{1-k}\langle \mu_A\circ\mu_A,\mu_{A_1}\circ\mu_{A_1}\rangle^{1/2}\langle \mu_A\circ\mu_A,\mu_{A_2}\circ\mu_{A_2}\rangle^{1/2} \\
&\leq 2^{2-k}\sigma.
\end{align*}
We may choose $k=O(1)$ so that this quantity is at most $2^{-5}\sigma$. By averaging, we then have 
\[\|\mu_V\ast \mu_X^{(k)}\ast F - \mu_X^{(k)}\ast F\|_\infty \leq 2^{-5}\sigma,\]
and thus 
\[\|\mu_V\ast\mu_X^{(k)}\ast \mu_A\circ\mu_A\circ(\mu_{A_1}\circ\mu_{A_2})\|_\infty \geq (1-2^{-4})\sigma.\]
We then have 
\[\|\mu_V\ast \mu_A\|_\infty \geq (1-2^{-4})\sigma\geq (1-2^{-4})(1+2^{-3})\geq 1+2^{-5},\]
establishing case (2).
\end{proof}

\textbf{Remark.} One may observe that this argument really gives a density increment of the form $\alpha\to \sigma\alpha$ for some $\sigma\geq 1+2^{-5}$, instead of just $\alpha \to (1+\Omega(1))\alpha$. Thus, in some cases, we obtain a density increment much greater than what is used here. This provides no quantitative savings in the finite field case, but will be useful in controlling the radius of Bohr sets in the integer case.

\section{The Integer Case}\label{mainsection}
The purpose of this section is to prove Theorem \ref{3APbound}. Throughout, we will let $G=\bbz/N\bbz$, where $N$ is an odd positive integer. Unlike in $\bbf_q^n$, $G$ may not have any nontrivial proper subgroups. To circumvent this issue, as usual, we will use regular Bohr sets as approximate subgroups relative to which we can perform a density increment argument. The preliminary facts about Bohr sets we will need are recorded in Appendix \ref{Bohrappendix}.

In order to obtain Theorem \ref{3APbound}, we will need to obtain careful control of the radius of the Bohr set on which we have a density increment. In turn, throughout this section, we will let $c>0$ be an absolute constant chosen to facilitate several applications of Lemmas \ref{narrowbohrsupport}, \ref{regularapproximation}, \ref{firstcompare}, and \ref{narrowdensity}. We will often say that an inequality is true provided $c$ is chosen small enough, and for the sake of concreteness, we observe that $c=1/(2^{13}\cdot 100)$ suffices for all such assertions.

We also define
\begin{definition} Let $B$ be a regular Bohr set with rank $d$. Given another Bohr set $B'$, we will say $B'\subset_r B$ if $B'$ is regular and $B' = B_\delta$ for some $\delta\in [r/2,r]$.
\end{definition}

Theorem \ref{3APbound} will follow from iteration of the following proposition, whose proof will occupy the bulk of the section.
\begin{proposition}\label{prop-it} Let $B$ be a regular Bohr set with rank $d$ and radius $\rho$. Let $A\subset B$ have relative density $\alpha$. Let $B^1\subset_{c/d}B$, and let $B_2\subset_{c/2d}B^1$. Then at least one of the following is true. 
\begin{enumerate}
    \item Many progressions: $A$ has at least $\frac{1}{4}\alpha^3N^2\mu(B^1)\mu(B^2)$ three-term arithmetic progressions.
    \item Density increment: There is a parameter $\sigma\in[1+2^{-12},\alpha^{-1}]$, a Bohr set $B'$ with rank at most $d+O(\lo{\alpha}^4)$ and radius at least $\Omega((c/2d\lo{\alpha})^{O(\log(\sigma))}\rho)$, and a translate $x$ such that the relative density of $(A+x)\cap B'$ in $B'$ is at least $\sigma\alpha$. 
\end{enumerate}
\end{proposition}

Before proving Proposition \ref{prop-it}, let us see how it implies Theorem \ref{3APbound}.
\begin{proof}[Proof of Theorem \ref{3APbound}, assuming Proposition \ref{prop-it}] As usual, it suffices to prove the result for $A\subset \bbz/N\bbz$ for $N$ an odd positive integer, since we may embed $A\subset \{1,\dots,N\}$ into $\bbz/(2N+1)\bbz$ without introducing any new three-term progressions.

Let $C>0$ be an absolute constant to be chosen later. Let $\alpha$ be the density of $A$ in $G$. Let $J$ be the maximal positive integer such that there exists a sequence $G=B^0\supset B^1\supset \dots \supset B^J$ of regular Bohr sets with ranks $0=d_0,\dots,d_J$ and radii $1=\rho_0,\dots,\rho_J$, translates $0=x_0,x_1,\dots,x_J$, and parameters $\sigma_1,\dots,\sigma_J\in [1+2^{-12},\alpha^{-1}]$ with the following properties.
\begin{itemize}
    \item For each $j\geq 1$, the relative density of $(A+x_j)\cap B^j$ in $B^j$ is at least $\sigma_j$ times the relative density of $(A+x_{j-1})\cap B^{j-1}$ in $B^{j-1}$.
    \item For each $j\geq 1$, $d_j\leq d_{j-1}+C\lo{\alpha}^4$.
    \item For each $j\geq 1$, $\rho_j\geq (c/2d\lo{\alpha})^{C\log(\sigma_j))}\rho_{j-1}$.
\end{itemize}

We can take $d_0=1$ and $\rho_0=1$.
By the first bullet point, we must have $\prod_{j=1}^J\sigma_j\leq \alpha^{-1}$, and so $J=O(\lo{\alpha})$. By the second bullet point, we must have $d_j\leq O(\lo{\alpha}^5)$ for all $j$. Then, by the third bullet point, $\rho_J\geq \prod_{j=1}^J(c/2\lo{\alpha})^{O(\log\sigma_j)}\geq \exp(-O(\lo{\alpha}\log(\lo{\alpha}))$.

Let $B'$, $B''$ be regular Bohr sets with $B''\subset_{c/2d}B'\subset_{c/2d}B^J$.We now apply Proposition \ref{prop-it} with $(A+x_J)\cap B^J$ in place of $A$, $B^J$ in place of $B$, $B'$ in place of $B^1$, and $B''$in place of $B^2$. If we are in case (2), then provided the constant $C$ is chosen large enough, this contradicts the maximality of $J$. In turn, we must be in case (1), so that the number of three-term progressions in $(A+x_J)\cap B^j$, and thus $A$, is at least
\[\frac{1}{4}\alpha^3N^2\mu(B')\mu(B'').\]
By Lemma \ref{bohrsizebound}, $\mu(B')\mu(B'') = \exp(-O(\lo{\alpha}^6\log(\lo{\alpha})))$. Since $A$ has no nontrivial three-term progressions, the number of three-term progressions in $A$ is $\alpha N$. We thus have the inequality
\[N\leq 4\alpha^{-2}\exp(O(\lo{\alpha}^6\log(\lo{\alpha})) = \exp(O(\lo{\alpha}^6\log(\lo{\alpha})),\]
which rearranges to the desired upper bound on $\alpha$.
\end{proof}

To prove Proposition \ref{prop-it}, we will follow the five-step plan discussed in section 2. We begin with the H\"older lifting and unbalancing step. Although our proof is similar to that in \cite{BS1}, it must be reworked slightly to provide the required control of the radius. 

We recall part of the unbalancing step fron \cite{BS1}.

\begin{lemma}[\cite{BS1}, Lemma 7]\label{unbalance} Let $\epsilon\in (0,1)$ and $\nu:G\to \mathbb{R}_{\geq 0}$ satisfy $\|\nu\|_1=1$ and $\hat{\nu}\geq 0$. Let $f:G\to\mathbb{R}$ be such that $\hat{f}\geq 0$. If $\|f\|_{p(\nu)}\geq \epsilon$ for some $p\geq 1$, then $\|f+1\|_{p'(\nu)}\geq 1+\frac{1}{2}\epsilon$ for some $p' = O_\epsilon(p)$. 
\end{lemma}

\begin{lemma}[H\"older lifting and unbalancing]\label{liftunbalance} Let $B$ be a regular Bohr set with rank $d$. Let $B'\subset B_{c/d}$ be a regular Bohr set. Suppose $A\subset B$ has relative density $\alpha$ and $C\subset B'$ has relative density $\gamma$. Then at least one of the following is true.
\begin{enumerate}
    \item Many progressions: $\langle \mu_A\ast\mu_A,\mu_C\rangle \geq 2^{-1}\mu(B)^{-1}$.
    \item Density increment: $\|\mu_A\ast \mu_{B'}\|_\infty \geq 2\mu(B)^{-1}$.
    \item $L^p$ discrepancy: There are Bohr sets $B''$, $B'''$ with $B'''\subset_{c/d}B''\subset_{c/d}B'$ such that, if $\nu = \mu_{B''}\circ\mu_{B''}\ast\mu_{B'''}\circ\mu_{B'''}$, there is a $p = O(\lo{\gamma})$ such that $\|\mu_A\circ\mu_A\|_{p(\nu)}\geq (1+2^{-5})\mu(B)^{-1}$.
\end{enumerate}
\end{lemma}

\begin{proof} Suppose cases (1) and (2) both do not hold. Then we have 
\[|\langle \mu_A\ast\mu_A,\mu_C\rangle - \mu(B)^{-1}|\geq 2^{-1}\mu(B)^{-1}.\]
Let $B^0 = B_{1+1/200d}$. Consider the decomposition
\[\mu_{A}\ast \mu_{A} = (\mu_{A}-\mu_{B^0}\ast \mu_{B'}) \ast (\mu_{A}-\mu_{B^0}\ast \mu_{B'}) + 2\mu_{A}\ast \mu_{B^0}\ast \mu_{B'} - \mu_{B^0}\ast \mu_{B'}\ast \mu_{B^0}\ast\mu_{B'}.\]
We estimate 
\[|\langle \mu_A\ast\mu_{B^0}\ast \mu_{B'},\mu_C\rangle - \mu(B)^{-1}| \leq |\langle \mu_A\ast\mu_{B^0}\ast \mu_{B'},\mu_C\rangle - \mu(B^0)^{-1}| + |\mu(B^0)^{-1}-\mu(B)^{-1}| = \]
\[|\langle \mu_A\ast\mu_{B^0}\ast\mu_{B'},\mu_C\rangle -\langle \mu_A\ast \mu_{B'},\mu_{B^0}\rangle| +O(c)\mu(B)^{-1}\leq \|\mu_A\ast\mu_{B'}\|_\infty \|\mu_{B^0}\ast\mu_C-\mu_{B^0}\|_1+O(c)\mu(B)^{-1}.\]
Since (2) does not hold, the first factor is at most $2\mu(B)^{-1}$. By Lemma \ref{narrowbohrsupport}, the second factor is $O(c)$. Provided $c$ is chosen small enough, we can ensure that the left hand side is bounded above by $2^{-5}\mu(B)^{-1}$. 

We may similarly estimate 
\begin{align*}
    |\langle \mu_{B^0}\ast \mu_{B'}\ast \mu_{B^0}\ast\mu_{B'},\mu_C\rangle - \mu(B)^{-1}| &\leq  |\langle \mu_{B^0}\ast \mu_{B'}\ast \mu_{B^0}\ast\mu_{B'},\mu_C\rangle  - \mu(B^0)^{-1}| + O(c)\mu(B)^{-1}
    \\&=|\langle \mu_{B^0}\ast \mu_{B'}\ast \mu_{B^0}\ast\mu_{B'},\mu_C\rangle  - \langle \mu_{B^0},\mu_{B^0}\rangle   | + O(c)\mu(B)^{-1}
    \\&\leq \|\mu_{B^0}\|_\infty \|\mu_{B^0}\ast\mu_{B'}\ast\mu_{B'}\ast\mu_C-\mu_{B^0}\|_1+O(c)\mu(B)^{-1}
    \\&\leq O(c)\mu(B^0)^{-1}+O(c)\mu(B)^{-1} = O(c)\mu(B)^{-1}.
\end{align*}
Provided $c$ is chosen small enough, we can ensure that the left hand side is bounded above by $2^{-4}\mu(B)^{-1}$.

Combining these estimates, we can conclude
\[|\langle (\mu_{A}-\mu_{B^0}\ast \mu_{B'}) \ast (\mu_{A}-\mu_{B^0}\ast \mu_{B'}),\mu_C\rangle|\geq 3\cdot 2^{-3}\mu(B)^{-1}. \]
By H\"older's inequality, for any $p_0\geq 1$, the left-hand side is bounded above by 
\[\gamma^{-1/p_0}\|(\mu_{A}-\mu_{B^0}\ast \mu_{B'}) \ast (\mu_{A}-\mu_{B^0}\ast \mu_{B'})\|_{p_0(\mu_{B'})}\geq 3\cdot 2^{-3}\mu(B)^{-1}.\]
We choose $p_0 = O(\lo{\gamma})$ such that $\gamma^{-1/p_0} \leq 3/2$, so \[\|(\mu_{A}-\mu_{B^0}\ast \mu_{B'}) \ast (\mu_{A}-\mu_{B^0}\ast \mu_{B'})\|_{p_0(\mu_{B'})} \geq 2^{-2}\mu(B)^{-1}.\]

We may now apply Lemma \ref{firstcompare}. Let $p_1$ be an even integer with $p_0\leq p_1< p_0+2$. We have 
\[\|(\mu_{A}-\mu_{B^0}\ast \mu_{B'}) \circ (\mu_{A}-\mu_{B^0}\ast \mu_{B'})\|_{p_1(\nu)}\geq 2^{-3}\mu(B)^{-1}.\]
Finally, we apply Lemma \ref{unbalance} to $f = \mu(B)(\mu_{A}-\mu_{B^0}\ast \mu_{B'}) \ast (\mu_{A}-\mu_{B^0}\ast \mu_{B'})$. We find an exponent $p\leq O(\lo{\gamma})$ such that 
\[\|(\mu_{A}-\mu_{B^0}\ast \mu_{B'}) \circ (\mu_{A}-\mu_{B^0}\ast \mu_{B'})+\mu(B)^{-1}\|_{p(\nu)}\geq \left(1+2^{-4}\right)\mu(B)^{-1}.\]
First, we observe that, provided $c$ is chosen small enough, $|\mu(B^0)^{-1}-\mu(B)^{-1}| \leq 2^{-7}\mu(B)^{-1}$. 

Moreover, we expand, using the symmetry of each Bohr set,
\[(\mu_{A}-\mu_{B^0}\ast \mu_{B'}) \circ (\mu_{A}-\mu_{B^0}\ast \mu_{B'}) + \mu(B^0)^{-1} \]\[\mu_A\circ\mu_A - 2(\mu_A\ast\mu_{B^0}\ast\mu_{B'}  - \mu(B^0)^{-1}) + (\mu_{B^0}\ast\mu_{B'}\ast\mu_{B^0}\ast\mu_{B'} - \mu(B^0)^{-1}).\]

We will now estimate $\|\mu_A\ast\mu_{B^0}\ast\mu_{B'}  - \mu(B^0)^{-1}\|_{p(\nu)}$. For any $x\in\supp(\nu)$, let $\delta_x$ denote the point mass at $x$. Since $A+B'\subset B^0$, we have the following identity:
\[|\mu_A\ast\mu_{B^0}\ast\mu_{B'}(x) - \mu(B^0)^{-1}| = |\langle \mu_A\ast\mu_{B'},\mu_{B^0}\ast\delta_x\rangle - \langle \mu_A\ast\mu_{B'},\mu_{B^0}\rangle| \leq \|\mu_A\ast\mu_{B'}\|_\infty\|\mu_{B^0}\ast\delta_x-\mu_{B^0}\|_1. \]
The first factor is at most $2\mu(B)^{-1}$, and the second factor is $O(c)$ by Lemma \ref{narrowbohrsupport}. Provided $c$ is chosen small enough, this quantity is at most $2^{-7}\mu(B)^{-1}.$
We thus have 
\[\|\mu_A\ast\mu_{B^0}\ast\mu_{B'}  - \mu(B^0)^{-1}\|_{p(\nu)}\leq \|\mu_A\ast\mu_{B^0}\ast\mu_{B'}  - \mu(B^0)^{-1}\|_{\infty (\nu)}\leq 2^{-7}\mu(B)^{-1}.\]

Finally, we estimate 
\[\|\mu_{B^0}\ast\mu_{B'}\ast\mu_{B^0}\ast\mu_{B'} - \mu(B^0)^{-1}\|_{p(\nu)}.\]

We have, for any $x\in\supp(\nu)$, 
\[|\mu_{B^0}\ast\mu_{B^0}\ast\mu_{B'}\ast\mu_{B'}(x) - \mu(B^0)^{-1}| = |\langle \mu_{B^0}\ast\mu_{B'}\ast\mu_{B'}\ast\delta_x,\mu_{B^0}\rangle - \langle \mu_{B^0},\mu_{B^0}\rangle |\]
\[\leq \|\mu_{B^0}\ast\mu_{B'}\ast\mu_{B'}\ast\delta_x-\mu_{B^0}\|_1\|\mu_{B^0}\|_\infty \leq O(c)\mu(B^0)^{-1} \leq 2^{-7}\mu(B)^{-1},\]
provided $c$ is small enough. Thus,
\[\|\mu_{B^0}\ast\mu_{B^0}\ast\mu_{B'}\ast\mu_{B'}- \mu(B^0)^{-1}\|_{p(\nu)}\leq \|\mu_{B^0}\ast\mu_{B^0}\ast\mu_{B'}\ast\mu_{B'}- \mu(B^0)^{-1}\|_{\infty(\nu)} \leq 2^{-7}\mu(B)^{-1}.\]
The inequality $\|\mu_A\circ\mu_A\|_{p(\nu)}\geq \left(1+2^{-5}\right)\mu(B)^{-1}$
thus follows from the triangle inequality.
\end{proof}

Now that we have completed the H\"older lifting and unbalancing steps, the next step is sifting. Like in the finite field case, we perform an iterated sifting argument. The following is analogous to Lemma \ref{ff-iteratedsifting}.

\begin{lemma}\label{iteratedsiftresult} Let $B$ be a regular Bohr set. Let $A\subset B$ have relative density $\alpha$. Let $B^1$, $B^2$ be a pair of Bohr sets with rank $d$ and $B^2\subset_{c/d}B^1$.
\[\|\mu_A\circ\mu_A\|_{p(\mu_{B^1}\circ\mu_{B^1}\ast\mu_{B^2}\circ\mu_{B^2})}\geq (1+2^{-5})\mu(B)^{-1}.\]

Then, there is a parameter $\sigma\in [1+2^{-7},\alpha^{-1}]$, there are regular Bohr sets $B'$, $B''$, there is a translate $t$, and there are subsets $A_1\subset B'$, $A_2\subset B''+t$ such that all the following are true.
\begin{enumerate}
    \item If $S = \{x:\mu_A\circ\mu_A\geq (1-2^{-10})\sigma\mu(B)^{-1}\}$, then $\langle \mu_{A_1}\circ\mu_{A_2},\mathds{1}_S\rangle \geq 1-2^{-12}$.
    \item The relative densities of $A_1$ in $B'$ and $A_2$ in $B''+t$ are each at least $\alpha^{p+O(\lo{\alpha})}$. 
    \item $B' = B^1_\delta$, where $\delta = (c/2d)^{O(\log(\sigma))}$ and $B''\subset_{c/d}B'$.
    \item For each $i\in\{1,2\}$, $\langle \mu_A\circ\mu_A,\mu_{A_i}\circ\mu_{A_i}\rangle \leq 2^8\sigma\mu(B)^{-1}$.
\end{enumerate}
\end{lemma}

In order to prove Lemma \ref{iteratedsiftresult}, we introduce a weighted variant of sifting. Note that in the following lemma, there is no Bohr set, and $\alpha$ is the absolute density of $A$ in $G$ instead of a relative density in a Bohr set. Thus, when applying Lemma \ref{weightedsift} to a subset $A\subset B$ with relative density $\alpha$, the absolute density $\alpha\cdot\mu(B)$ will appear in the resulting inequalities.
\begin{lemma}\label{weightedsift}
	 	Let $p\geq 1$ be an integer. Let $C_1,C_2\subset G$ have densities $\gamma_1$ and $\gamma_2$, respectively.  Let $\mu=\mu_{C_1}\circ\mu_{C_2}$.  Let $A\subset G$ have density $\alpha$. Let $f:G\to \mathbb{R}$ be a function. 
		
	 	Given $s = (s_1,\dots,s_p)\in G^p$, let $A_1(s) = C_1\cap (A+s_1)\cap\dots\cap (A+s_p)$ and define $A_2(s)$ similarly. Let $\alpha_1(s)$ and $\alpha_2(s)$ denote their densities in $G$.

	 	Then 
	 	\[\langle (\mu_A\circ\mu_A)^p,f\rangle_\mu = |G|^{-p}\alpha^{-2p}\gamma_1^{-1}\gamma_2^{-1}\sum_{s\in G^p} \alpha_1(s)\alpha_2(s)\langle \mu_{A_1(s)}\circ\mu_{A_2(s)},f\rangle.\]

        In particular, 
        \[\mathbb{P}(s)=\frac{\alpha_1(s)\alpha_2(s)}{|G|^p\alpha^{2p}\gamma_1\gamma_2\|\mu_A\circ\mu_A\|_{p(\mu)}^p}\]
        is a probability measure on $G^p$.
	 \end{lemma}

     \begin{proof} We have
    \begin{align*}
        \langle (\mu_A\circ\mu_A)^p,f\rangle_\mu &= |G|^{-2}\gamma_1^{-1}\gamma_2^{-1}\sum_{c_1\in C_1,c_2\in C_2}(\mu_A\circ\mu_A(c_1-c_2))^pf(c_1-c_2) 
        \\&=|G|^{-2}\gamma_1^{-1}\gamma_2^{-1}|G|^{-p}\sum_{c_1\in C_1,c_2\in C_2}f(c_1-c_2)\left(\sum_{t\in G}\mu_A(c_1+t)\mu_A(c_2+t)\right)
        \\&=\gamma_1^{-1}\gamma_2^{-1}|G|^{-2-p}\alpha^{-2p}\sum_{c_1\in C_1,c_2\in C_2}f(c_1-c_2)\sum_{s\in G^p}\mathds{1}_{A_1(s)}(c_1)\mathds{1}_{A_2(s)}(c_2)
        \\&=|G|^{-p}\alpha^{-2p}\gamma_1^{-1}\gamma_2^{-1}\sum_{s\in G^p} \alpha_1(s)\alpha_2(s)\langle \mu_{A_1(s)}\circ\mu_{A_2(s)},f\rangle.
    \end{align*}

     The final claim follows by setting $f=1$.
     \end{proof}

\begin{lemma}\label{weightedsiftresult}
Let $A$ be a subset of a regular Bohr set $B$ with relative density $\alpha$. Abbreviate $\beta = \frac{|B|}{|G|}$. Let $C_1,C_2\subset G$ have densities $\gamma_1,\gamma_2$, respectively. Let $\mu = \mu_{C_1}\circ\mu_{C_2}$. Let $p\geq 1$ be an integer, let $\epsilon\in(0,1)$, and suppose $\|\mu_A\circ\mu_A\|_{p(\mu)} \geq \beta^{-1}$. Let 
\[S = \{x:\mu_A\circ\mu_A(x)\leq (1-\epsilon)\|\mu_A\circ\mu_A\|_{p(\mu)}\}.\] 
Then there are sets $A_1\subset C_1$, $A_2\subset C_2$, each with relative density at least $\frac{1}{4}\alpha^p$, such that 
\[\langle \mu_{A_1}\circ\mu_{A_2},\mathds{1}_{G\setminus S}\rangle \geq 1-4(1-\epsilon)^p.\]
\end{lemma}
\begin{proof} Let $\P,$ $A_1(s)$, $A_2(s)$, $\alpha_1(s),$ and $\alpha_2(s)$ be as in Lemma \ref{weightedsift}.
Let $f = \mathds{1}_S$. We compute 
\[  \sum_{s\in G^p}\P(s)\langle \mu_{A_1(s)}\circ\mu_{A_2(s)},f\rangle = \frac{\langle (\mu_A\circ\mu_A)^p,f\rangle_{\mu}}{\|\mu_A\circ\mu_A\|^p_{p(\mu)}} \leq (1-\epsilon)^p. \]
We also compute 
\begin{align*}
    \sum_{s\in G^p}\P(s)\alpha_1^{-1}(s) &\leq  \frac{1}{|G|^p\beta^{2p}\alpha^{2p}\gamma_1\gamma_2\|\mu_A\circ\mu_A\|_{p(\mu)}^p}\sum_{s\in G^p}\alpha_2(s) \\
    &=\frac{1}{\beta^p\alpha^p\gamma_1\|\mu_A\circ\mu_A\|_{p(\mu)}^p} \\
    &\leq \alpha^{-p}\gamma_1^{-1},
\end{align*}
where we used $\|\mu_A\circ\mu_A\|_{p(\mu)}\geq \beta^{-1}$ for the final inequality.
By a similar argument, 
\[    \sum_{s\in G^p}\P(s)\alpha_2^{-1}(s)\leq \alpha^{-p}\gamma_2^{-1}.\]
By Markov's inequality, there is an $s\in G^p$ such that 
\[\langle \mu_{A_1(s)}\circ\mu_{A_2(s)},f\rangle \leq 4(1-\epsilon)^p,\]
and for each $i\in\{1,2\}$, $\alpha_i(s)\geq \frac{1}{4}\alpha^p\gamma_i$. For each $i\in\{1,2\}$, the relative density of $A_i$ in $C_i$ is $\alpha_i/\gamma_i\geq \frac{1}{4}\alpha^p$. The inequality 
\[\langle \mu_{A_1}\circ\mu_{A_2},\mathds{1}_{G\setminus S}\rangle \geq 1-4(1-\epsilon)^p\]
follows from observing $\mathds{1}_{G\setminus S}=1-f$.
\end{proof}

The following lemma is required to pass from lower bounds on $\langle \mu_A\circ\mu_A,\mu_{C}\circ\mu_{C}\rangle$ to lower bounds on $\langle \mu_A\circ\mu_A,\mu_{A_1}\circ\mu_{A_2}\rangle$, where $A_2$ is a subset of a narrower Bohr set than $A_1$. This will be required to perform the almost-periodicity step later on.

\begin{lemma}\label{iteratedsifting} Let $A$ be a subset of a regular Bohr set $B$ with relative density $\alpha$. Let $B^1$ be a regular Bohr set with rank $d$. Let $C$ be a subset of $B^1$ with relative density $\gamma$. Let $\sigma\geq \mu(B)^{-1}$ be a parameter.
Suppose 
\[\langle \mu_A\circ\mu_A,\mu_C\circ\mu_C\rangle \geq 2^7\sigma.\]
Then there are regular Bohr sets $B'$, $B''$, where $B'=B^1_\rho$ for some $\rho\geq (c/2d)^2$ and $B''\subset_{c/d}B'$, a translate $t$, and subsets $A_1\subset B'$, $A_2\subset B''+t$, each with relative density at least $\min(\alpha^{O(1)}\cdot \gamma,\alpha^{\lo{\alpha}/\log(2)+O(1)}) $, such that 
\[\langle \mu_A\circ\mu_A,\mu_{A_1}\circ\mu_{A_2}\rangle \geq \sigma.\]
\end{lemma}

\begin{proof} Let $B^2\subset_{c/d}B^1$, and let $B^0 = B^1_{1+1/200d}$. We abbreviate $\mu(B)=\beta$, $\mu(B^1) = \beta_1$, and $\mu(B^2)=\beta_2$. Let $p =\max(\lo{\gamma}/\lo{\alpha}, \lo{\alpha}/\log(2))+O(1)$ be such that $p$ is a positive integer and $2^{2-p}\leq 1/2$. We divide the proof into two cases:
\begin{enumerate}
    \item $\|\mu_A\circ\mu_A\|_{p(\mu_{B^2+t})}\leq 8\sigma$ for all $t\in G$.
        \item $\|\mu_A\circ\mu_A\|_{p(\mu_{B^2+t})}\geq 8\sigma$ for some $t\in G$.
\end{enumerate}
\textbf{Case (1).} 
For each $y\in B^0$, let $\gamma(y)$ be the relative density of $C\cap (B^2+y)$ in $B^2+y$. Let $C(y) = C\cap (B^2+y)$. 

For each $y\in B^0$, we have 
\begin{equation*}
\begin{split}
    \langle \mu_A\circ\mu_A,\mu_{C(y)}\circ\mu_C\rangle  &= \langle \mu_A\circ\mu_A\ast\mu_C,\mu_{C(y)}\rangle  \\ &=\beta_2 \langle \mu_A\circ\mu_A\ast\mu_C,\mu_{C(y)}\rangle_{\mu_{B^2+y}} \\ 
     &\leq \beta_2\|\mu_A\circ\mu_A\ast\mu_C\|_{p(\mu_{B^2+y})}\|\mu_{C(y)}\|_{\frac{p}{p-1}(B^2+y)} \\&= \gamma(y)^{-1/p}\|\mu_A\circ\mu_A\ast\mu_C\|_{p(\mu_{B^2+y})} \\
    &= \gamma(y)^{-1/p}\left\|\frac{1}{|C|}\sum_{c\in C}\mu_A\circ\mu_A(\cdot + c)\right\|_{p(\mu_{B^2+y})} \\&\leq \gamma(y)^{-1/p}\frac{1}{|C|}\sum_{c\in C}\|\mu_A\circ\mu_A\|_{p(\mu_{B^2+y-c})} \\ &\leq \gamma(y)^{-1/p}\cdot 8\sigma. 
\end{split}
\end{equation*}

Observe that $\mu_C\ast \mu_{B^2}$ is a probability measure with support contained in $B^0$, and that for any $y\in B^0$, $\mu_C\ast \mu_{B^2}(y) = \gamma(y)\gamma^{-1}\beta_1^{-1}$. We compute
\begin{align*}
\E_{y\in G} \gamma(y)\gamma^{-1}\beta_1^{-1}\langle \mu_A\circ\mu_A,\mu_{C(y)}\circ\mu_C\rangle &= \E_{y\in G} \gamma(y)\gamma^{-1}\beta_1^{-1}\langle \mu_A\circ\mu_A\ast\mu_C,\mu_{C(y)}\rangle  \\
&=\E_{y\in G}\mathds{1}_{B^0}(y)\gamma(y)\gamma^{-1}\beta_1^{-1}\langle \mu_A\circ\mu_A\ast\mu_C,\mu_{C(y)}\rangle \\&= 
\E_{y\in G}\mathds{1}_{B^0}(y)\langle \mu_A\circ\mu_A\ast\mu_C,\mu_C\mu_{B^2+y}\rangle  
\\
&=\E_{y\in G}\mathds{1}_{B^0}(y)\E_{x\in G}(\mu_A\circ\mu_A\ast\mu_C(x))\mu_C(x)\mu_{B^2}(x-y) \\&= \E_{x\in G}(\mu_A\circ\mu_A\ast\mu_C(x))\mu_C(x)(\mathds{1}_{B^0}\ast\mu_{B^2}(x)).
\\&\geq \E_{x\in G}(\mu_A\circ\mu_A\ast\mu_C(x))\mu_C(x)\mathds{1}_{B^1}(x)
\\&= \E_{x\in G}(\mu_A\circ\mu_A\ast\mu_C(x))\mu_C(x)
\\&=\langle \mu_A\circ\mu_A\ast\mu_C,\mu_C\rangle =\langle \mu_A\circ\mu_A\ast\mu_C,\mu_C\rangle \geq 2^7\sigma,
\end{align*}
where in the first inequality we applied Lemma \ref{regularapproximation}. We now compute, on the other hand,
\begin{align*}
    \E_{y\in G}\mathds{1}_{\gamma(y)\leq \alpha\gamma}\gamma(y)\gamma^{-1}\beta_1^{-1}\langle \mu_A\circ\mu_A,\mu_{C(y)}\circ\mu_C\rangle  &= 
    |G|^{-1}\sum_{y\in B^0}\mathds{1}_{\gamma(y)\leq \alpha\gamma}\gamma(y)\gamma^{-1}\beta_1^{-1}\langle \mu_A\circ\mu_A,\mu_{C(y)}\circ\mu_C\rangle \\ &\leq  
    |G|^{-1}\sum_{y\in B^0}\mathds{1}_{\gamma(y)\leq \alpha\gamma} \gamma(y)^{1-1/p}\gamma^{-1}\beta_1^{-1}\cdot 8\sigma  \\& \leq 
    8\sigma|G|^{-1}\sum_{y\in B^0} \gamma^{-1/p}\alpha^{1-1/p} \beta_1^{-1} \\&= 8\sigma\cdot \frac{|B^0|}{|B^1|}\cdot \gamma^{-1/p}\alpha^{1-1/p}.   
\end{align*}
By the regularity of $B^1$, we have $\frac{|B^0|}{|B^1|}\leq 2$. By the choice of $p$, we have $\gamma^{-1/p}\alpha^{1-1/p} \leq 4$, so this quantity is at most $2^6\sigma$.

We then have
\[\E_{y\sim \mu_C\ast\mu_{B^2}} \mathds{1}_{\gamma(y)\geq \alpha\gamma} \langle \mu_A\circ\mu_A,\mu_{C(y)}\circ\mu_C\rangle \geq 2^6\sigma.\]
We may thus conclude that there is a $t\in G$ such that $\gamma(t)\geq \alpha\gamma$ and $\langle \mu_A\circ\mu_A,\mu_{C}\circ\mu_{C(t)}\rangle \geq 2^6\sigma\geq \sigma$. The claim thus holds with $B'=B^1$, $B'' = B^2$, $A_1=C$, and $A_2 = C\cap (B_2+t)$. 

\bigskip

\textbf{Case (2).} We now assume that $\|\mu_A\circ\mu_A\|_{p(\mu_{B^2+t})}\geq 8\sigma$ for some $t\in G$. 

Let $B'$, $B''$ be such that $B''\subset_{c/d}B'\subset_{c/d}B^2$. 
By Lemma \ref{firstcompare}, $\|\mu_A\circ\mu_A\|_{p(\mu_{B'}\ast\mu_{B'}\ast\mu_{B''}\ast\mu_{B''})}\geq 4\sigma$. By averaging, there is a $t\in G$ such that $\|\mu_A\circ\mu_A\|_{p(\mu_{B'}\circ\mu_{B''+t})}\geq 4\sigma$. 

We apply Lemma \ref{weightedsiftresult} with $B'$ in place of $C_1$, $B''+t$ in place of $C_2$, and $1/2$ in place of $\epsilon$. We locate subsets $A_1\subset B'$, $A_2\subset B''+t$, each with relative densities at least $\frac{1}{4}\alpha^p$, such that 
\[\langle \mu_{A_1}(s)\circ\mu_{A_2}(s),\mathds{1}_{\{x:\mu_A\circ\mu_A(x)\geq 2\sigma\}}\rangle \geq 1-4\cdot 2^{-p}\geq 1/2.\]
We have $\frac{1}{4}\alpha^p\geq \min(\alpha^{O(1)}\cdot \gamma,\alpha^{\lo{\alpha}/\log(2)+O(1)}) $ by the choice of $p$, and 

\[\langle \mu_{A_1}(s)\circ\mu_{A_2},\mu_A\circ\mu_A\rangle \geq \sigma\]
since $\mu_A\circ\mu_A\geq 2\sigma\mathds{1}_{x:\mu_A\circ\mu_A(x)\geq 2\sigma}$.
\end{proof}

We are now ready to prove Lemma \ref{iteratedsiftresult}.
\begin{proof}[Proof of Lemma \ref{iteratedsiftresult}] By averaging, let $t_1\in G$ be such that 
\[\|\mu_A\circ\mu_A\|_{p(\mu_{B^1}\circ\mu_{B^2+t_1})}\geq (1+2^{-5})\mu(B)^{-1}.\]
We replace $p$ by a quantity $p+O(1)$ large enough so that $p$ is an integer and $(1-2^{-7})^p\leq 2^{-9}$. 

We apply Lemma \ref{weightedsiftresult} with $B^1$ in place of $C_1$, $B^2+t_1$ in place of $C_2$, and $2^{-7}$ in place of $\epsilon$. We obtain sets $Z_1$ and $Z_2$ with relative densities at least $\frac{1}{4}\alpha^p$, in $B^1$ and $B^2+t_1$, respectively, and such that 
\[\langle \mu_{Z_1}\circ\mu_{Z_2},\mathds{1}_{\{x:\mu_A\circ\mu_A(x)\geq (1-2^{-7})(1+2^{-5})\mu(B)^{-1}\}}\rangle \geq 1-4(1-2^{-7})^p\geq 1-2^{-7}.\]
It follows that 
\[\langle \mu_A\circ\mu_A,\mu_{Z_1}\circ\mu_{Z_2}\rangle \geq (1-2^{-7})(1-2^{-7})(1+2^{-5})\mu(B)^{-1}\geq (1+2^{-7})\mu(B)^{-1}.\]
Let $c_0$ be an absolute constant to be chosen later. Let $J$ be the largest positive integer such that the following is true. There is a collection $B^1, B^2, B^3, \dots,  B^{2J}$ of regular Bohr sets, a collection of translates $t_1,\dots,t_J$, and sets $Z_1,\dots,Z_{2J}$ with $Z_j\subset B^j$ for $j$ odd and $Z_j\subset B^j+t_{j/2}$ for $j$ even with the following properties.
\begin{enumerate}
    \item For any odd $j\geq 3$, $B^j = B^{j-2}_\delta$ for some $\delta \geq (c/2d)^3$, and $B^{j+1}\subset_{c/d}B^j$.
    \item For each $j$, let $\zeta_j$ denote the relative density of $Z_j$ in $B^j$ for $j$ odd, or $B^j+t_{j/2}$ for $j$ even. For odd $j\geq 3$, $\zeta_j\geq \min(\alpha^{c_0}\zeta_{j-1},\alpha^{c_0}\zeta_{j-2},\alpha^{\lo{\alpha}/\log(2)+c_0})$, and for even $j\geq 4$, $\zeta_j\geq \min(\alpha^{c_0}\zeta_{j-2},\alpha^{c_0}\zeta_{j-3},\alpha^{\lo{\alpha}/\log(2)+c_0})$.
    \item For each odd $j\geq 3$,
    \[\langle \mu_A\circ\mu_A,\mu_{Z_j}\circ\mu_{Z_{j+1}}\rangle \geq 2\langle \mu_A\circ\mu_A,\mu_{Z_{j-2}}\circ\mu_{Z_{j-1}}\rangle.\]
\end{enumerate}
Observe that 
$\langle \mu_A\circ\mu_A,\mu_{Z_1}\circ\mu_{Z_2}\rangle\geq \mu(B)^{-1}$ and $\langle \mu_A\circ\mu_A,\mu_{Z_j}\circ\mu_{Z_{j+1}}\rangle \leq \alpha^{-1}\mu(B)^{-1}$ for all $j$. Thus, by (3), we must have $J= O(\lo{\alpha})$. 

We shall prove the claim with parameter $\sigma = \mu(B)\langle \mu_A\circ\mu_A,\mu_{Z_{2J-1}}\circ\mu_{Z_{2J}}\rangle$, $B'=B^{2J-1}$, $B''=B^{2J}$, and $t=t_J$. Observe that with this choice of $\sigma$, $J = O(\lo{\sigma})$. Let $q = O(1)$ be large enough so that $(1-2^{-10})^q\leq 2^{-14}$. Observe that 
\[\|\mu_A\circ\mu_A\|_{q(\mu_{Z_{2J-1}}\circ\mu_{Z_{2J}})} \geq \|\mu_A\circ\mu_A\|_{1(\mu_{Z_{2J-1}}\circ\mu_{Z_{2J}})} = \langle \mu_A\circ\mu_A,\mu_{Z_{2J-1}}\circ\mu_{Z_{2J}}\rangle = \sigma\mu(B)^{-1}.\]

To obtain the sets $A_1$ and $A_2$, we apply Lemma \ref{weightedsiftresult} with $Z^{2J-1}$ in place of $C_1,$ $Z^{2J}$ in place of $C_2$, $q$ in place of $p$, and $2^{-10}$ in place of $\epsilon$. We obtain sets $A_1\subset Z_{2J-1}$, $A_2\subset Z_{2J}$, such that 
\[\frac{|A_1|}{|Z_{2J-1}|} \geq \frac{1}{4}\alpha^q,\quad \frac{|A_2|}{|Z_{2J}|}\geq \frac{1}{4}\alpha^q,\quad\text{and}\]
\[\langle \mu_{A_1}\circ\mu_{A_2},\mathds{1}_{S}\rangle \geq 1-4\cdot (1-2^{-10})^q \geq 1-2^{-12}.\]
This last inequality verifies conclusion (1). To verify conclusion (2), let $\alpha_1,\alpha_2$ denote the relative densities of $A_1$ in $B'$, $A_2$ in $B''+t$, respectively. We have $\alpha_1=\alpha^{O(1)}\cdot \zeta_{2J-1}$ and $\alpha_2=\alpha^{O(1)}\cdot \zeta_{2J}$. Using the inequalities
\[\zeta_1,\zeta_2\geq \alpha^{p+O(1)},  \quad \zeta_j\geq \min(\alpha^{c_0}\zeta_{j-1},\alpha^{c_0}\zeta_{j-2},\alpha^{c_0}\zeta_{j-3},\alpha^{\lo{\alpha}/\log(2)+c_0}),\quad\text{and}\quad J=O(\lo{\alpha}),  \]
we may conclude $\alpha_i \geq \alpha^{p+O(\lo{\alpha})}$ for $i\in\{1,2\}$. 
To verify conclusion (3), observe that $J = O(\lo{\sigma})$, and $B' = B^1_{\delta}$ for $\delta = (c/d)^{O(J)}$. 

Finally, we must verify conclusion (4). Assume for a contradiction that this fails, and for some $i\in\{1,2\}$, 
\[\langle \mu_A\circ\mu_A,\mu_{A_i}\circ\mu_{A_i}\rangle \geq 2^8\sigma\mu(B)^{-1}.\]
Assume first that $i=1$. We apply Lemma \ref{iteratedsifting} with $2\sigma\mu(B)^{-1}$ in place of $\sigma$, $A_1$ in place of $C$, and $B^{2J-1}$ in place of $B^1$. We obtain a pair of regular Bohr sets, $B^{2J+1}, B^{2J+2}$, with $B^{2J+1}= B^{2J-1}_\delta$ for some $\delta\geq (c/2d)^2$ and $B^{2J+2}\subset_{c/d}B^{2J+1}$. Provided $c_0$ is sufficiently large, we obtain subsets $Z_{2J+1}\subset B^{2J+1}$ and $Z_{2J+2}\subset B^{2J+2}$ such that $\zeta_{2J+1}$ and $\zeta_{2J+2}$ are each at least $\min(\alpha^{c_0}\zeta_{2J-1},\alpha^{\lo{\alpha}/\log(2)+c_0})$. The sets $Z_{2J+1},Z_{2J+2}$ also satisfy
\[\langle \mu_A\circ\mu_A,\mu_{Z_1}\circ\mu_{Z_2}\rangle \geq 2\sigma\mu(B)^{-1}.\]
This contradicts the maximality of $J$.

The case $i=2$ is very similar; the only changes are as follows. First, $B^{2J+1}=B^{2J}_\delta$ for some $\delta\geq (c/2d)^2$, but since $B^{2J}\subset_{c/d}B^{2J-1}$, we have $B^{2J+1}=B^{2J-1}_\delta$ for some $\delta\geq (c/2d)^3$. Second, $\zeta_{2J+1}$ and $\zeta_{2J+2}$ are each at least $\min(\alpha^{c_0}\zeta_{2J},\alpha^{\lo{\alpha}/\log(2)+c_0})$. This also contradicts the maximality of $J$.\end{proof}

With the sifting step complete, we now need a version of Croot-Sisask almost-periodicity, and a version of Chang's lemma to bootstrap it. The following almost-periodicity result is \cite[Theorem 5.1]{SS}.
\begin{theorem}\label{AP} Let $\epsilon\in(0,1)$ and let $k\in\bbn$ be parameters. Let $A,M,L,S\subset G$. Suppose $|A+S|\leq K|A|$ and $\eta = |M|/|L|\leq 1$. Then there is a set $X\subset S$ with $|X|\geq \exp(-O(k^2\epsilon^{-2}\lo{\eta}\log(2K))|S|$ such that 
\[\|\mu_X^{(k)}\ast\mu_A\ast\mu_M\ast\mathds{1}_L - \mu_A\ast\mu_M\ast\mathds{1}_L\|_\infty \leq \epsilon.\]
\end{theorem}

\noindent Like in the previous section, for a set $X\subset G$, we define 
\[\Delta_{1/2}(X) = \{\gamma\in\hat{G}:|\hat{\mu_X}(\gamma)|\geq 1/2\}.\]
The following variant of Chang's lemma is \cite[Proposition 5.3]{SS}.
\begin{lemma}\label{Chang} Let $\nu\in(0,1]$.
Let $B$ be a regular Bohr set with rank $d$ and radius $\rho$. Let $X\subset B$ with relative density $\xi$.
Then there is a regular Bohr set $B'$ with rank at most $d+O(\lo{\xi})$ and radius at least $\Omega(\rho\nu/d^2\lo{\xi})$ such that $|1-\gamma(t)|\leq \nu$ for all $\gamma\in\Delta_{1/2}(X)$ and $t\in B'$.

\end{lemma}

Finally, with all these preliminary results recorded, we are ready to prove Proposition \ref{prop-it}. At the beginning, for technical reasons, we pass from $A$ and $\alpha$ to the slightly sparser $A'$ and smaller $\alpha'$. This is unrelated to the other details of the argument and the reader may choose to ignore it on a first reading. 
\begin{proof}[Proof of Proposition \ref{prop-it}] We begin by applying Lemma \ref{narrowdensity} with $B^1$ in place of $B'$, $B^2$ in place of $B''$, and $\epsilon = 2^{-13}$. If we are in case (2), then we have established the density increment with $\sigma = 1+2^{-12}$. Otherwise, assume we are in case (1), so there is a translate $y$ such that for each $i\in\{1,2\}$, the relative density of $A+y\cap B^i$ in $B^i$ is at least $(1-2^{-11})\alpha$. We will let $A' = A+y\cap B^1$ and $\alpha'$ be the relative density of $A'$ in $B^1$. 

Observe that $2\cdot B^2$ is a regular Bohr set with the same rank and radius as $B^2$, and $2\cdot B^2\subset B^1_{c/d}$ by the traingle inequality. We may thus apply Lemma \ref{liftunbalance} with $B^1$ in place of $B$, $2\cdot B^2$ in place of $B'$, and $2\cdot (A+y\cap B^2)$ in place of $C$. If we are in case (1), then there are at least 
\[(1-2^{-11})^3\frac{1}{2}\alpha^3N^2\mu(B^1)\mu(B^2)\geq \frac{1}{4}\alpha^3N^2\mu(B^1)\mu(B^2)\]
arithmetic progressions in $A$, so we are done. If we are in case (2), we have a density increment with $\sigma= 2(1-2^{-11})$. Let us thus assume that we are in case (3), so that there are Bohr sets $B^3$ and $B^4$ with $B^4\subset_{c/d}B^3\subset_{c/d}(2\cdot B^2)$ such that, if $\nu = \mu_{B^3}\circ\mu_{B^3}\ast\mu_{B^4}\circ\mu_{B^4}$, then there is a $p = O(\lo{\alpha})$ such that 
\[\|\mu_{A'}\circ\mu_{A'}\|_{p(\nu)}\geq (1+2^{-5})\mu(B^1)^{-1}.\]

We now apply Lemma \ref{iteratedsiftresult} with $A'$ in place of $A$, $B^1$ in place of $B$, $B^3$ in place of $B^1$, and $B^4$ in place of $B^2$. We obtain a paramter $\sigma'\in [1+2^{-7},(\alpha')^{-1}]$, regular Bohr sets $B^5$, $B^6$, a translate $t$, and subsets $A_1\subset B^5$, $A_2\subset B^6+t$ such that all the following are true.
\begin{enumerate}
    \item If $S = \{x:\mu_{A'}\circ\mu_{A'}\geq (1-2^{-10})\sigma'\mu(B^1)^{-1}\}$, then $\langle \mu_{A_1}\circ\mu_{A_2},\mathds{1}_S\rangle \geq 1-2^{-12}$.
    \item The relative densities of $A_1$ in $B^5$ and $A_2$ in $B^6+t$ are each at least $\alpha^{O(\lo{\alpha})}$. 
    \item $B^5 = B^3_\delta$, where $\delta = (c/2d)^{O(\log(\sigma'))}$ and $B^6\subset_{c/d}B^5$.
    \item For each $i\in\{1,2\}$, $\langle \mu_{A'}\circ\mu_{A'},\mu_{A_i}\circ\mu_{A_i}\rangle \leq 2^8\sigma'\mu(B^1)^{-1}$.
\end{enumerate}

From condition (1), we have $\mu_{A_1}\ast\mu_{-A_2}\ast\mathds{1}_{-S}(0) \geq 1-2^{-12}$. Let $B^7\subset_{c/d}B^6.$ We apply Theorem \ref{AP} with $\epsilon = 2^{-12}$, $k$ a parameter to be chosen later, $-A_2$ in place of $A$, $A_1$ in place of $M$, $-S$ in place of $L$, and $B^7$ in place of $S$ (with apologies for the clash in notation). Without loss of generality, we can take $S$ to be suppoprted in $B^5+B^6+t$. Observe that \[|A_2+B^7|\leq |B^6+B^7|\leq 2|B^6|\leq \alpha^{-O(\lo{\alpha})}|A_2|\]
and 
\[\frac{|A_1|}{|S|} \geq \frac{|A_1|}{2|B_5|}\geq \alpha^{O(\lo{\alpha})}.\]
We thus obtain a set $X\subset B^7$ with relative density at least $\exp(-O(k^2\lo{\alpha}^4))$ such that 
\[\|\mu_X^{(k)}\ast\mu_{A_1}\circ\mu_{A_2}\circ\mathds{1}_S - \mu_{A_1}\circ\mu_{A_2}\circ\mathds{1}_S\|_\infty \leq 2^{-12}.\]
We thus have 
\[\mu_X^{(k)}\ast\mu_{A_1}\circ\mu_{A_2}\circ\mathds{1}_S(0)\geq 1-2^{-11}.\]
Since $\mu_{A'}\circ\mu_{A'}\geq (1-2^{-10})\sigma'\mu(B^1)^{-1}$ pointwise on $S$, we obtain
\[\mu_X^{(k)}\ast\mu_{A_1}\circ\mu_{A_2}\ast\mu_{A'}\circ\mu_{A'}(0) \geq (1-2^{-10})(1-2^{-11})\sigma'\mu(B^1)^{-1}\geq (1-2^{-9})\sigma'\mu(B^1)^{-1}.\]

We apply Lemma \ref{Chang} with $X\subset B^7$ and $\nu>0$ a parameter to be chosen later. We obtain a regular Bohr set $B^8$ with rank at most $d+O(k^2\lo{\alpha}^4)$ and radius at least $\rho \cdot (c/2d)^{O(\lo{\sigma'})}\cdot \nu/d^2k^2\lo{\alpha}^4$ (here, we are using that the radius of $B^7$ is at least $\rho\cdot (c/2d)^{O(\lo{\sigma'})}$) such that 
\[|1-\gamma(x)|\leq \nu\text{ for all }\gamma\in\Delta_{1/2}(X)\text{ and }x\in B^8.\]
We abbreviate $F=\mu_{A'}\circ\mu_{A'}\ast\mu_{A_1}\circ\mu_{A_2}$. 

For any $x\in B^8$, we have

\begin{align*}    
\|\mu_X^{(k)}\ast F (\cdot + x) - \mu_X^{(k)}\ast F\|_\infty &\leq \sum_{\gamma\in\hat{G}} |\hat{\mu_X}(\gamma)|^k|\hat{F}(\gamma)||\gamma(x)-1|\\
&=\sum_{\gamma\in\Delta_{1/2}(X)} |\hat{\mu_X}(\gamma)|^k|\hat{F}(\gamma)||\gamma(x)-1|+\sum_{\gamma\notin\Delta_{1/2}(X)} |\hat{\mu_X}(\gamma)|^k|\hat{F}(\gamma)||\gamma(x)-1| 
\\&\leq (\nu+2^{1-k})\sum_{\gamma\in\hat{G}}|\hat{F}(\gamma)|
\\&\leq (\nu+2^{1-k})\left(\sum_{\gamma\in\hat{G}}|\hat{\mu_{A'}}(\gamma)|^2|\hat{\mu_{A_1}}(\gamma)|^2\right)^{1/2}\left(\sum_{\gamma\in\hat{G}}|\hat{\mu_{A'}}(\gamma)|^2|\hat{\mu_{A_2}}(\gamma)|^2\right)^{1/2}
\\&=(\nu+2^{1-k})\langle \mu_{A'}\circ\mu_{A'},\mu_{A_1}\circ\mu_{A_1}\rangle^{1/2}\langle \mu_{A'}\circ\mu_{A'},\mu_{A_2}\circ\mu_{A_2}\rangle^{1/2}
\\&\leq (\nu+2^{1-k})2^8\sigma'\mu(B^1)^{-1}.
\end{align*}
We can choose $\nu = \Omega(1)$ and $k = O(1)$ so that the above quantity is at most $2^{-9}\sigma'\mu(B^1)^{-1}.$ By averaging, we have 
\[\|\mu_{B^8}\ast\mu_X^{(k)}\ast F-\mu_X^{(k)}\ast F\|_\infty \leq 2^{-9}\sigma'\mu(B^1)^{-1},\]
so 
\[\|\mu_{B^8}\ast\mu_X^{(k)}\ast F\|_\infty \geq (1-2^{-8})\sigma'\mu(B^1)^{-1}.\]
We thus have 
\[\|\mu_{B^8}\ast\mu_{A'}\|_\infty \geq (1-2^{-8})\sigma'\mu(B^1)^{-1},\]
so there is a translate $z$ of $A'$ such that the relative density of $(A'+z)\cap B^8$ in $B^8$ is at least $(1-2^{-8})\sigma'\alpha' \geq (1-2^{-8})(1-2^{-11})\sigma'\alpha$. In turn, the relative density of $(A+y+z)\cap B^8$ in $B^8$ is also at least $(1-2^{-8})(1-2^{-11})\sigma'\alpha$.
Since $(1-2^{-8})(1-2^{-11})\sigma'\geq (1-2^{-8})(1-2^{-11})(1+2^{-7})\geq 1+2^{-10}$, we have obtained the density increment with $B'=B^8$ and $\sigma = (1-2^{-8})(1-2^{-11})\sigma'$.
\end{proof}

\begin{appendices}
    \section{Bohr Sets}\label{Bohrappendix}
In this appendix, we collect the facts about Bohr sets required to carry out our argument. For a more detailed introduction to Bohr sets, one may consult \cite{TV}, Chapter 4. 

\begin{definition}[Bohr set] For $\Lambda\subset \hat{G}$ and $\rho\geq 0$, we define the Bohr set $B$ with frequency set $\Lambda$ and radius $\rho$ to be 
\[\{x\in G:|1-\gamma(x)|\leq \rho\text{ for all }\gamma\in\Lambda\}.\]
The quantity $|\Lambda| $ is the rank of $B$.
\end{definition}

When we speak of a Bohr set, we are implicitly referring to the triple $(\Lambda,\rho,B)$, and thus it makes sense to consider, for $\delta>0$, the dilate $B_\delta$, which is the Bohr set with frequency set $\Lambda$ and radius $\delta\rho$. However, when we write an inclusion $B'\subset B$ of two Bohr sets, we refer only to the inclusion as sets, and do not impose any relation between the corresponding frequencies or radii.

\begin{definition}[Regular Bohr set] A Bohr set $B$ of rank $d$ is regular if for all $|\kappa|\leq 1/100d$, we have 
\[(1-100d|\kappa|)|B|\leq |B_{1+\kappa}|\leq (1+100d|\kappa|)|B|.\]
\end{definition}

\begin{lemma} For any Bohr set $B$ there is a $\rho\in [1/2,1]$ such that $B_\rho$ is regular.
\end{lemma}
\begin{proposition}\label{bohrsizebound} If $\rho\in(0,1)$ and $B$ is a Bohr set of rank $d$, then $|B_\rho|\geq (\rho/4)^d|B|$.
\end{proposition}

The following lemma is a slight generalization of \cite{BS1}, Lemma 27.

\begin{lemma}\label{narrowbohrsupport} Suppose $B$ is a regular Bohr set of rank $d$. Let $\rho,\delta>0$ be such that $\rho+\delta<1/100d$. Let $\mu:G\to \mathbb{R}_{\geq 0}$ be a probability measure supported on $B_\rho$. Then 
\[\|\mu_{B_{1+\delta}}\ast \mu - \mu_{B_{1+\delta}}\|_1\leq 200(\rho+\delta)d.\]
\begin{proof}
    By the triangle inequality,
    \[\E_{x\in G}|\mu_{B_{1+\delta}}\ast \mu (x)-\mu_{B_{1+\delta}}(x)|  \leq  \E_{y\in G}\mu(y)\E_{x\in G}|\mu_{B_{1+\delta}}(x-y)-\mu_{B_{1+\delta}}(x)| = \E_{y\in G}\mu(y)\frac{|(y+B_{1+\delta})\Delta B_{1+\delta}|}{|B_{1+\delta}|},\]
    where $\Delta$ is the symmetric set difference. By the definition of regularity, since $\rho+\delta<1/100d$, this quantity is at most $200(\rho+\delta)d$. 
\end{proof}
\end{lemma}

The following is \cite{BS1}, Lemma 28.

\begin{lemma}\label{regularapproximation} Let $B$ be a regular Bohr set of rank $d$ and let $L\geq 1$ be any integer. If $\nu:G\to\mathbb{R}_{\geq 0}$ is supported on $LB_\rho$, where $\rho\leq 1/100Ld$ and $\|\nu\|_1=1$, then 
\[\mathds{1}_B\leq \mathds{1}_{B_{1+L\rho}} \ast\nu,\text{ and}\]
\[\mu_B\leq 2\mu_{B_{1+L\rho}}\ast\nu.\]
\end{lemma}
\begin{proof} For any $x\in B$ and $y\in \supp(\nu)$, we have $\mathds{1}_{B_{1+L\rho}}(x-y) = 1$. The first inequality thus follows by averaging. The second inequality follows from the first and that $\frac{|B_{1+L\rho}|}{|B|}\leq 2$ due to the regularity of $B$. 
\end{proof}

The following is a slight generalization of \cite{BS1}, Proposition 19.

\begin{lemma}\label{firstcompare}Let $p\geq 2$ be an even integer. Let $f:G\to\mathbb{R}$, let $t\in G$, let $B\subset G$, and $B',B''\subset B_{1/400\text{rk}(B)}$ all be regular Bohr sets. Then \[\|f\circ f\|_{p(\mu_{B'}\circ\mu_{B'}\ast\mu_{B''}\circ\mu_{B''})}^p\geq \frac{1}{2}\|f\ast f \|_{p(\mu_{B+t})}^p.\]
Also,
\[\|f\circ f\|_{p(\mu_{B'}\circ\mu_{B'}\ast\mu_{B''}\circ\mu_{B''})}^p\geq \frac{1}{2}\|f\circ f \|_{p(\mu_{B+t})}^p.\]
\end{lemma}
\begin{proof} We will prove the first inequality; the proof of the second is identical. Let $\nu = \mu_{B'}\circ\mu_{B'}\ast\mu_{B''}\circ\mu_{B''}$, and note that $\hat{\nu}\geq 0$. We apply Lemma \ref{regularapproximation} with $L=4$ to obtain 
\[\mu_B\leq 2\mu_{B_{1+1/100d}}\ast\nu.\]
Translating by $t$, we have 
\[\mu_{B+t}\leq 2\mu_{B_{1+1/100d}+t}\ast\nu.\]
We then write
\begin{align*}
    \|f\ast f \|_{p(\mu_{B+t})}^p &= \E_{x\in G}(f\ast f)^p(x)\mu_{B+t}(x) \\ 
    &\leq 2\E_{x\in G}(f\ast f)^p(x)(\mu_{B_{1+1/100d}+t}\ast \nu(x))\\
    &=2\E_{y\in B_{1+1/100d}+t}\E_{x\in G}(f\ast f)^p(x)\nu(x-y).
\end{align*}
By averaging, there is a $y\in G$ such that 
\[\frac{1}{2}\|f\ast f \|_{p(\mu_{B+t})}^p\leq \E_{x\in G}(f\ast f)^p(x)\nu(x-y).\]
Taking the Fourier transform of the right-hand side, we have 
\begin{align*}
    \E_{x\in G}(f\ast f)^p(x)\nu(x-y) &= \sum_{\gamma\in\hat{G}} \hat{\nu}(\gamma)\gamma(-y)(\hat{f}^2\underbrace{\ast\dots\ast}_{p\text{ times}}\hat{f}^2)(\gamma) \\
    &\leq \sum_{\gamma\in\hat{G}} \hat{\nu}(\gamma)(|\hat{f}|^2\underbrace{\ast\dots\ast}_{p\text{ times}}|\hat{f}|^2)(\gamma) \\
    &= \|f\circ f\|_{p(\nu)}^p. \qedhere
\end{align*}
\end{proof}

    The following lemma is standard, although it is traditionally proven for $\delta \leq \alpha^{O(1)}$, which is insufficient for our purposes. In \cite[Lemma 12]{S}, it is proved for significantly larger $\delta$, which helps us save in the radius during iteration.
    
    \begin{lemma}\label{narrowdensity} Let $B$ be a regular Bohr set of rank $d$ and suppose $A\subset B$ has relative density $\alpha$. Let $\epsilon>0$. Let $\delta>0$ be so that $|B|\geq (1-\epsilon)|B_{1+\delta}|$. Let $B',B''\subset B_\delta$. Then, either:
    \begin{enumerate}
        \item There is an $x$ such that $\mu_A\ast\mu_{B'}(x)\geq (1-4\epsilon)\mu(B)^{-1}$ and $\mu_A\ast\mu_{B''}(x)\geq (1-4\epsilon)\mu(B)^{-1}$, or
        \item $\|\mu_A\ast\mu_{B'}\|_\infty\geq (1+2\epsilon)\mu(B)^{-1}$ or $\|\mu_A\ast\mu_{B''}\|_\infty\geq (1+2\epsilon)\mu(B)^{-1}$. 
    \end{enumerate}

    \end{lemma}
    Note that by the definition of a regular Bohr set, the above lemma is effective for $\delta  \approx c/d$, instead of just $\delta \approx c\alpha/d$.  
\end{appendices}

\printbibliography

\end{document}